\documentclass[12pt,reqno]{amsart}

 \usepackage[utf8]{inputenc}
 \usepackage[T1]{fontenc}
 \usepackage[usenames, dvipsnames]{color}
 \usepackage{ulem}
 
 \usepackage{dsfont, amsfonts, amsmath, amssymb,amscd, stmaryrd, latexsym, amsthm, dsfont}
 \usepackage[frenchb,english]{babel}
 \usepackage{enumerate}
 \usepackage{longtable}
 \usepackage{geometry}
 \usepackage{tikz}

 \usepackage{hyperref}
 \hypersetup{
 	colorlinks=true,
 	urlcolor=blue,
 	citecolor=blue}

 \usepackage{tabularx}
 \usepackage{float}
 \usepackage{tikz}
 \usepackage{xypic}
 \usetikzlibrary{shapes,arrows}
 \usepackage{amsmath}
 \makeatletter
 \def\legendre@dash#1#2{\hb@xt@#1{%
 		\kern-#2\p@
 		\cleaders\hbox{\kern.5\p@
 			\vrule\@height.2\p@\@depth.2\p@\@width\p@
 			\kern.5\p@}\hfil
 		\kern-#2\p@
 }}
 \def\@legendre#1#2#3#4#5{\mathopen{}\left(
 	\sbox\z@{$\genfrac{}{}{0pt}{#1}{#3#4}{#3#5}$}%
 	\dimen@=\wd\z@
 	\kern-\p@\vcenter{\box0}\kern-\dimen@\vcenter{\legendre@dash\dimen@{#2}}\kern-\p@
 	\right)\mathclose{}}
 \newcommand\legendre[2]{\mathchoice
 	{\@legendre{0}{1}{}{#1}{#2}}
 	{\@legendre{1}{.5}{\vphantom{1}}{#1}{#2}}
 	{\@legendre{2}{0}{\vphantom{1}}{#1}{#2}}
 	{\@legendre{3}{0}{\vphantom{1}}{#1}{#2}}
 }
 \def\dlegendre{\@legendre{0}{1}{}}
 \def\tlegendre{\@legendre{1}{0.5}{\vphantom{1}}}
 \makeatother

 \newtheorem{Lemma}{Lemma}

 \newtheorem{Theorem}{Theorem}
 \newtheorem{corollary}{Corollary}
 
 \newtheorem{Remark}{Remark}

 \usepackage{multicol}

 \def\QQ{\mathbb{Q}}
 
 \def\ZZ{\mathds{Z}}

 \subjclass[2010]{ 11R29; 11R23;   11R18; 11R20.}
 \keywords{Cyclotomic $Z_p$-extension, $2$-rank; $2$-class group.}

\begin{document}
 	
 	\title[$2$-class groups of cyclotomic towers]{$2$-class groups of cyclotomic towers of imaginary biquadratic fields and applications}
 	\author[M. M. Chems-Eddin]{	MOHAMED MAHMOUD CHEMS-EDDIN}
 	\address{Mohamed Mahmoud Chems-Eddin: Mohammed First University, Mathematics Department, Sciences Faculty, Oujda, Morocco }
 	\email{2m.chemseddin@gmail.com}
 	
 	\author[K. Müller]{KATHARINA M\"ULLER}
 	\address{Katharina Müller: Georg-August-Universität, Göttingen, Germany}
 	\email{katharina.mueller@mathematik.uni-goettingen.de}

 	\maketitle
 	\begin{abstract}
 		Let $d$ be an odd positive square-free integer. In this paper we shall investigate the structure of the $2$-class group of the cyclotomic $\mathbb{Z}_2$-extension of the   imaginary biquadratic number field $\mathbb{Q}(\sqrt{d},\sqrt{-1})$ if $d$ is of specifiic form. Furthermore, we deduce the structure of the $2$-class group of the cyclotomic $\mathbb{Z}_2$-extension of $\mathbb{Q}(\sqrt{-d})$.
 	\end{abstract}
 	\section{Introduction} 
 	Let $p$ be a prime number and  $k$ be a   number field. 
 	Denote by 	$k_\infty$ the cyclotomic $\mathbb Z_p$-extension of $k$.
 	The field $k_\infty$ contains a unique cyclic subfield $k_n$  of degree $p^n$ over $k$. The field $k_n$ is called the $n$-th	layer of the $\mathbb Z_p$-extension of $k$. 
 	In $1959$, the study of      $p$-class numbers of number fields with large degree led to  a spectacular result due to Iwasawa, that we shall recall here and use later (for $p=2$).
 	Denote by $e_n$ the highest power of $p$ dividing the class number of $k_n$. 
 	Then there exist integers $\lambda$,  $\mu\geq 0$ and  $\nu$, all independent of $n$, and  an integer $n_0$ such that:
 	\begin{eqnarray}\label{iwasawa}e_n=\lambda n+\mu p^n+\nu,\end{eqnarray}
 	for all $n\geq n_0$. The integers $\lambda$,  $\mu\geq 0$ and  $\nu$ are called the Iwasawa invariants of $k_\infty$(cf. \cite{iwasawa59}).
 	

 	Thereafter, the study of cyclotomic  $\mathbb{Z}_2$-extensions of $\mathrm{CM}$-Fields 
 	was the subject of many   papers and is still of huge interest in algebraic number theory. Let us mention a few papers related to the subject of this paper. In 1980, Kida studied   the   Iwasawa  $\lambda^-$-invariants  and  the  2-ranks  of  the 
 	narrow ideal  class  groups  in  the  2-extensions  of  $\mathrm{CM}$-fields (cf. \cite{kida}). In  $2018$,  Atsuta (cf. \cite{atsuta}) studied the maximal finite submodule of the minus part of the Iwasawa module attached to $k_{\infty}$, while the second named author worked on the capitulation in the minus-part   in the steps of the cyclotomic $\ZZ_p$-extension  of a $\mathrm{CM}$-field $k$ (cf. \cite{Katharina}).
 	
 	In this paper we will concentrate on $\mathrm{CM}$-fields of the following form: Let $n\geq 0$ be a natural number, $d$ be an odd positive square-free integer and $L_{n,d}:=\mathbb{Q}(\zeta_{2^{n+2}}, \sqrt{d})$. In $2019$,  the first named author, Azizi  and Zekhnini, computed the rank of the $2$-class group of $L_{n,d}$, the layers of the  $\mathbb{Z}_2$-extension of some special Dirichlet fields of the form $L_{0,d}=\mathbb{Q}(\sqrt{d},\sqrt{-1})$ (cf. \cite{ACZ,chemsZkhnin2,chems}). Li, Ouyang, Xu and Zhang computed the $2$-class groups of these fields for $d$ being a prime congruent to $3\pmod 8$, $5\pmod 8$ and $7\pmod {16}$ (cf. \cite{LiouYangXuZhang}).
 	
 	In the present work we consider some different infinite families    of biquadratic   fields $L_{0,d}$ and determine the structure of the $2$-class group of the $n$-th layer of their cyclotomic $\ZZ_2$-extensions. Let $h_2(d)$ denote the $2$-class number of the quadratic field $\mathbb{Q}(\sqrt{d})$. The main aim of this paper is to prove the following Theorem using  some new techniques based on Iwasawa theory.
 	\begin{Theorem}Let $d$ be a positive  square-free integer and $n\ge 1$ be an integer.  
 		\begin{enumerate}[\rm 1.]
 			\item Assume $d$ has one of the following forms: 
 			\begin{enumerate}[\rm $\bullet$]
 				\item 	$d=p$, for a prime  $p\equiv 9\pmod{16}$ such that $(\frac{2}{p})_4=1$,
 				\item    $d=pq$, for two primes $p\equiv q\equiv 3 \pmod{8}$.
 			\end{enumerate}
 			Then the $2$-class group of $L_{n,d}$ is isomorphic to $\ZZ/2^{n+r-2}\ZZ\times\ZZ/2\ZZ$, for a constant $r$   such that  $2^r=h_2(-2d)$.
 			\item Let $d=pq$, for  two primes $p$ and $q$ such that $p\equiv-q\equiv 5  \pmod 8$. Then the $2$-class group of $L_{n,d}$ is isomorphic to $\ZZ/2^{n+r-2}\ZZ$, for a constant $r$  such that and $2^r=2\cdot h_2(-pq)$. Further, Greenberg's Conjecture holds for the field $L_{n,d}^+$, i.e., the $2$-class number   of $L_{n,d}^+$ is uniformly bounded.
 		\end{enumerate}	
 	\end{Theorem}

 	The plan of this paper is the following: In section 2  we will summarize some results on minus parts of $2$-class groups of  $\mathrm{CM}$-fields. In section 3 we collect results on the rank of the $2$-class groups of the fields $L_{n,d}$ and prove some of the main ingredients for the proof of the main Theorem. Section 4 contains   the proof of the main Theorem (cf. Theorems 8 and 9) and finally in   section 5, we apply our main results to give the $2$-class groups of the layers of the cyclotomic $\ZZ_2$-extension of some imaginary quadratic fields. The cyclotomic $\ZZ_2$-extension of imaginary quadratic fields were already investigated by Mizusawa in \cite{mizuzawa}. As applications of our first above result we give a more precise description of the structure of the $2$-class groups of the cyclotomic $\ZZ_2$-extensions for  certain families of imaginary quadratic fields (cf. Theorem \ref{thm 10} and Theorem \ref{thm11}):
 	\begin{Theorem}
 		Let $d$ be a positive square-free integer and $n\ge 1$. Let $K_{0,d}=\QQ(\sqrt{-d})$ and define the field $K_{n,d}$ as the $n$-th layer of the cyclotomic $\ZZ_2$-extension of $K_{0,d}$.
 		\begin{enumerate}[\rm 1.]
 			\item Let $d$ have one of the following forms
 			\begin{enumerate}[\rm $\bullet$]
 				\item $d=pq$ for two primes $p\equiv q\equiv 3\pmod 8$,
 				\item $d=p$ for a prime $p\equiv 9\pmod {16}$ such that $(\frac{2}{p})_4=1$.
 				
 			\end{enumerate} Let $2^r=h_2(-2d)$. Then the $2$-class group of $K_{n,d}$ is isomorphic to $\ZZ/2\ZZ\times \ZZ/2^{n+r-1}\ZZ$.
 			\item Assume that $d=pq$ is the product of two primes $p\equiv -q\equiv 5\pmod 8$ and let $2^r=2\cdot h_2(-pq)$. Then the $2$-class group of $K_{n,d}$ is isomorphic to $\ZZ/2^{n+r-1}\ZZ$.
 		\end{enumerate}
 	\end{Theorem}

 	\section*{Notations}
 	Let $k$ be a $\mathrm{CM}$ number field.	The next notations will be used for the rest of this article:
 	{  \begin{itemize}
 			\item $d$:  An odd positive square-free integer,
 			\item $n$: An integer $\geq 0$,
 			\item $K_n=\mathbb{Q}(\zeta_{2^{n+2}})$,
 			\item $K_n^+$: The maximal real subfield of  $K_n$,
 			\item $L_{n,d}=K_n(\sqrt{d})$,
 			\item $L_{n,d}^+$:   The  maximal real subfield of $L_{n,d}$,
 			\item $\tau$: A topological generator of $Gal(L_{\infty,d}/L_{0,d})$,	
 			\item $\Lambda=\ZZ_p[[T]]$ for $T=\tau-1$,
 			\item $\omega_n=(T+1)^{2^n}-1$,
 			\item $\nu_{n,m}=\omega_n/\omega_m$ for $n> m\ge 0$,
 			\item $\mu(M)$, $\lambda(M)$: The Iwasawa invariants introduced in \eqref{iwasawa} for a $\Lambda$-torsion module $M$,
 			\item $\lambda(k)=\lambda(A_{\infty}(k))$ (a precise definition of $A_{\infty}$ for the base field $k$ is given in Section \ref{sec:pre}), if $k=L_{0,d}$ we will also write $\lambda$ for $\lambda(L_{0,d})$,
 			\item $\lambda^-(k)=\lambda(A_{\infty}^-(k))$ (a precise definition of $A_{\infty}^-$ for the base field $k$ is given in Section \ref{sec:pre}), if $k=L_{0,d}$ we will also write $\lambda^-$ for $\lambda^-(L_{0,d})$,
 			\item $\mu^-(k)=\mu(A^-_{\infty}(k))$,
 			\item $\mathcal{N}$:  The application norm for the extension $L_{n,d}/K_n$,
 			\item $E_{k}$:   The unit group of  $k$,
 			\item $\mathrm{Cl}(k)$:  The class group of $k$,
 			\item $\mu_{k}$: The number of roots of unity contained in $k$,
 			\item $W_k$: The group of roots of unity in $k$,
 			\item $h_2(d)$:  The $2$-class number of the quadratic field $\mathbb{Q}(\sqrt{d})$,
 			\item $\genfrac(){}{0}{a}{p}_4$: The   biquadratic residue symbol,
 			\item $\left( \frac{\alpha,d}{\mathfrak p}\right)$:  The quadratic  norm residue symbol 	for $L_{n,d}/K_n$,
 			\item $Q_k$: Hasse's unit index of a CM-field $k$,
 			\item $q(L_{1,d}):=(E_{L_{1,d}}: \prod_{i}E_{k_i})$, with  $k_i$ are  the  quadratic subfields	of $L_{1,d}$.	

 	\end{itemize}}

 	\section{Some preliminary results on the minus part of the $2$-class group}
 	\label{sec:pre}
 	Let $p$ be a prime and $k$ be an arbitrary $\mathrm{CM}$-field containing the $p$-th roots of unity (the $4$-th roots of unity if $p=2$).
 	Consider the cyclotomic $\ZZ_p$-extension of $k$, denoted by $k_{\infty}$.
 	The complex conjugation of $k$, denoted by $j$, acts on  $A_n$, the $p$ part of the class group  of the intermediate fields $k_n$,  as well as on the projective limit $A_{\infty}=\lim_{\infty\leftarrow n} A_n$.  Usually one defines the minus part of the class group as $\widehat{A^-_n}=\{a\in A_n\mid ja=-a\}$ and the plus part as $\widehat{A^+_n}=\{a\in A_n\mid ja=a\}$. For $p\neq 2$ this yields a direct decomposition of $A_n= \widehat{ A_n^-}\oplus \widehat{ A_n^+}$. Further, it is well known that there is no capitulation on the minus part for $p\neq 2$. For $p=2$ this is in general not true. To avoid this problem we define $A_n^+$ as the group of strongly ambiguous classes with respect to the extension $k_n/k_n^+$ and $A_n^-=A_n/A_n^+$. Note that $A_{\bullet}^+=\widehat{A_{\bullet}^+}$ and $A_{\bullet}^-\cong \widehat{A_{\bullet}^-}$ for $p\neq 2$ (see \cite{Katharina}).For the rest of the paper we will study only cyclotomic $\ZZ_2$-extensions. 
 	
 	Note that the projective limit $A_{\infty}^-=\lim_{\infty\leftarrow n}A_n^-$ is a finitely generated $\Lambda$-torsion module. 
 	
 	\begin{Lemma}
 		\label{lem:rankbounded}
 		Assume that $\mu(A_{\infty}^-)=0$. Then there exists some $n_0\ge 0$ such that we have $\lambda(A_{\infty}^-)\ge 2\textup{-rank}(A_n^-)$ for all $n\ge n_0$.
 	\end{Lemma}
 	\begin{proof}
 		By  \cite[Theorem 2.5]{Katharina} there is no finite submodule in $A_{\infty}^-$. So if $\mu(A_{\infty}^-)=0$ the $2$-rank  and  $\lambda$-invariant of $A_{\infty}^-$  are equal. Thus, the claim is immediate for $n_0$ being the index such that all primes above $2$ are totally ramified in $K_{\infty}/K_{n_0}$.
 	\end{proof}
 	\begin{Remark}
 		If $K=L_{0,d}=\mathbb{Q}(\sqrt{-1},\sqrt{d})$, then $K_{\infty}/K$ is totally ramified and $n_0=0$.
 	\end{Remark}
 	\begin{Lemma}
 		\label{equiv}
 		Assume that $\mu(A_{\infty}^-)=0$. Then $\lambda(A^-_{\infty})=\lambda(\widehat{A^-_{\infty}})$.
 	\end{Lemma}
 	\begin{proof}
 		Note that $2A_{\infty}\subset (1+j)A_{\infty}+(1-j)A_{\infty}\subset A_{\infty}$. Clearly, all elements in $(1+j)A_{\infty}$ are strongly ambiguous. Thus, if we consider the projection \[\pi: A_{\infty}\to A_{\infty}^-\] we see that $(1+j)A_{\infty}$ lies in the kernel of $\pi$. On the other hand $j(1-j)a=-(1-j)a$. So if a class in $(1-j)A_{\infty}$ is strongly ambiguous then it is of order dividing $2$. As $\mu=0$ we obtain that $(1-j)A_{\infty}$ intersects the kernel of $\pi$ only in a finite submodule.  It follows that \[\lambda(A^-_{\infty})=\lambda((1-j)A_{\infty}).\] Note that $2\widehat{A^-_{\infty}}\subset (1-j){A^-_{\infty}}\subset \widehat{A^-_{\infty}}$. Hence, we see that \[\lambda(A_{\infty}^-)=\lambda(\widehat{A^-_{\infty}}).\]
 	\end{proof}
 	\section{Preliminaries on the fields $L_{n,d}$ and $L_{n,d}^+$}
 	
 	To determine the structure of the $2$-class group along a cyclotomic tower the $\lambda$-invariants of $A_n$ are of particular interest. Kida proved the following formula.
 	\begin{Theorem}\cite[Theorem 3]{kida}
 		Let $F$ and $K$ be   $\mathrm{CM}$-fields and $K/F$ a finite $2$ extension. Assume that $\mu^-(F)=0$. Then 
 		\[\lambda^-(K)-\delta(K)=[K_{\infty}:F_{\infty}]\left(\lambda^-(F)-\delta(F)\right)+\sum (e_{\beta}-1)-\sum(e_{\beta^+}-1),\]
 		where $\delta(k)$ takes the values $1$ or $0$ according to whether  $k_\infty$ contains the fourth roots of unity or not. The $e_{\beta}$ is the ramification index of a prime $\beta$ in $K_{\infty}/F_{\infty}$ coprime to $2$ and $e_{\beta^+}$ is the ramification index for a prime coprime to $2$ in $K_{\infty}^+/F_{\infty}^+$.
 	\end{Theorem}
 	Note that Kida proves results for $\lambda(\widehat{A^-})$. But due to Lemma \ref{equiv} this $\lambda$-invariant equals the $\lambda$-invariant of $A^-$.
 	\begin{Theorem}\label{thm 3}
 		\label{lambda}
 		Assume that $d$ is the product of    $r$ primes   congruent to $7$ or $9 \pmod{16}$ and $s$ primes   congruent to $3$ or $5\pmod 8$. 
 		Then $$\lambda^-(L_{0,d})=2r+s-1.$$
 	\end{Theorem}
 	\begin{proof}
 		Let $K=L_{0,d}=\mathbb{Q}(\sqrt{d},\sqrt{-1})$ and $F=\mathbb{Q}(\sqrt{-1})$. Then $\delta(F)=\delta(K)=1$ and $\lambda^-(F)=0$. Every prime congruent to $7$ or $9$ modulo $16$ splits into $4$ primes in $K_n$ for $n$ large enough, while it splits only into $2$ primes in $K^+_n$ (see \cite[Proposition 1]{chems}). Primes congruent to $3$ or $5$ modulo $8$ decompose into $2$ primes in $K_n$, while $K_n^+$ contains only one prime above $p$  (see \cite[Proposition 2]{chemsZkhnin2}). As $[K_{\infty}:F_{\infty}]=[K_{\infty}^+:F_{\infty}^+]=2$ all the non trivial terms satisfy $e_{\beta}=e_{\beta^+}=2$. Plugging all of this into Kida's formula we obtain
 		\[\lambda^-(L_{0,d})-1=2(0-1)+4r+2s-2r-s=2r+s-2\] and the claim follows.
 	\end{proof}
 	The above result gives $\lambda^-(L_{0,d})$ of some fields $L_{0,d}$. Noting that $\lambda^+(L_{0,d})$ is related to the class numbers of the real  fields $L_{n,d}^+$, we need the following theorem:
 	\begin{Theorem}\label{thm 4}
 		Let $d$ be an odd positive  square-free integer and $n\geq 1$.
 		Then, the class number of $L_{n,d}^+$ is odd if and only if $d$ takes one of the following forms 
 		\begin{enumerate}[\rm 1.]
 			\item $d=q_1q_2$ with $q_i\equiv 3\pmod 4$ and $q_1$ or $q_2\equiv 3\pmod8$.
 			\item $d$ is a prime $p$ congruent to   $3\pmod 4$.
 			\item $d$ is a prime $p$ congruent to   $5\pmod 8$.
 			\item $d$ is a prime $p$ congruent to $1\pmod 8$ and $(\frac{2}{p})_4(\frac{p}{2})_4=-1$.
 		\end{enumerate}	
 	\end{Theorem}
 	\begin{proof}
 		The extension $L_{n+1,d}^+/L_{n,d}^+$ is a quadratic extension that ramifies at the prime ideals of $L_{n,d}^+$ lying over $2$ and is unramified elsewhere for all $n\geq 1$. Let $H(L_{n,d}^+) $ be the $2$-Hilbert class field of $L_{n,d}^+$ and $X_n$ its Galois group over $L_{n,d}^+$. 
 		Let $Y$ be  the $\Lambda$-submodule of $X_{\infty}=\lim_{\infty\leftarrow n}X_n$ such that $X_0\cong X_{\infty}/Y$. Then $X_n\cong X_{\infty}/\nu_{n,0}Y$ \cite[Lemma 13.18]{washington1997introduction}. In particular, if $X_n$ is trivial then $X_{\infty}=\nu_{n,0}X_{\infty}$ and $X_{\infty}$ is trivial by Nakayama's Lemma. Hence,  the class number of $L_{1,d}^+$ being odd implies that the class number of $L_{n,d}^+$ is odd. 
 		The converse follows from \cite[Theorem 10.1]{washington1997introduction} and the fact that     the extension $L_{n+1,d}^+/L^+_{n,d}$ is totally ramified. Hence, the class number of $L_{n,d}^+$ is odd if and only if the class number of  $L_{1,d}^+=\mathbb{Q}(\sqrt{2},\sqrt{d})$ is odd. See \cite[pp. 155, 157]{connor88} and \cite[p. 78]{froh} for the rest.
 	\end{proof}

 	\begin{Theorem}
 		\label{cyclic}
 		Let $d>2$ be  an odd square-free integer and $n\geq 2$  a positive integer. Then the $2$-class group of $L_{n,d}$
 		is  cyclic non-trivial if and only if  $d$ takes one of the following forms:
 		\begin{enumerate}[\rm1.]
 			\item $d$ is a prime congruent to $  7\pmod{16}$,
 			\item $d=pq$, where $p$ and $q$ are two primes such that  $q\equiv 3\pmod 8$ and $p\equiv 5\pmod 8$.
 		\end{enumerate}
 	\end{Theorem}
 	\begin{proof}
 		By \cite[Theorem 6]{chemsZkhnin2}, it suffices to check the case when $d=p$ is a prime congruent to $7\pmod{8}$. We shall distinguish two cases.
 		\begin{enumerate}[$\bullet$]
 			\item Suppose that $p$ is congruent to $15\pmod{16}$ and let $\sigma$ denote its  Frobenius homomorphism in $Gal(\mathbb{Q}(\zeta_{16})/\mathbb{Q})$. Then $\sigma(\zeta_{16})=\zeta_{16}^p$ by the definition of the  Frobenius homomorphism. Let $H$ be the group generated by $\sigma$. Then $p$ is totally split in $\mathbb{Q}(\zeta_{16})^H/\mathbb{Q}$. Since $p\equiv 15\mod{16}$, $\sigma$ is the complex conjugation. Hence, $p$ is totally split in $\mathbb{Q}(\zeta_{16})^+/\mathbb{Q}$ and inert in $\mathbb{Q}(\zeta_{16})/\mathbb{Q}(\zeta_{16})^+$.
 			
In particular, there are $4$ primes of $K_2$ lying over $p$. Then, by the ambiguous class number formula (cf. \cite[Lemma 2.4]{Qin2}) $2$-rank$(\mathrm{Cl}(L_{2,d}))=4-1-e$, where  $e$ is defined by $2^{ e}=[E_{K_2}:E_{K_2} \cap  \mathcal{N}(L_{2,d}^*)]$. The unit group of $K_2$ is given by $E_{K_2}=\langle\zeta_{16},\xi_{3}, \xi_{5}, \xi_{7} \rangle$, where $\xi_{k}=\zeta_{16}^{(1-k)/2}\frac{1-\zeta_{16}^k}{ 1-\zeta_{16}}$. Let $\mathcal{N}'$ be the norm form $K_2$ to $K_2^+$. Since $p$ is inert in $K_2/K_2^+$ we obtain for $k=3, 5$ or $7$
 			$$\left( \frac{\xi_{k},p}{\mathfrak p_{K_{2}}}\right)= \left( \frac{\mathcal{N}'(\xi_{k}),p}{\mathfrak p_{K_{2}^+}}\right)=\left( \frac{\xi_{k}^2,p}{\mathfrak p_{K_{2}^+}}\right)=1.$$
 			
 			Then $e$ is at most  $1$. So $2$-rank$(\mathrm{Cl}(L_{2,d}))\geq 4-1-1=2$. Hence, the $2$-class group of $L_{n,d}$ is not cyclic.
 			\item 	Suppose now that    $p$ is congruent to  $7\pmod{16}$, then by   Theorem \ref{thm 4},   the   class number   of $L_{n,d}^+$ odd. Hence, $\lambda=\lambda^-$. Since   the primes above $2$ are unramified in $L_{n,p}/L_{n,p}^+$ for $n$ large enough all strongly ambigous ideals in $L_{n,d}$ are actually ideals from $L_{n,d}^+$ and the $2$-rank of $A_n$ is bounded by $\lambda$. By \cite[Theorem 4.4]{ACZ}, the $2$-class group of $L_{1,p}$ is cyclic non-trivial and by Theorem \ref{lambda} $\lambda^-=1$. Which completes the proof.
 		\end{enumerate}
 	\end{proof}	
 	\begin{Theorem}
 		Assume that $d$ takes one of the forms of Theorem \ref{cyclic}.
 		Then $\lambda(L_{0,d})=1$ and  Greenberg's conjecture holds for $L_{n,d}^+$.
 	\end{Theorem}
 	\begin{proof}
 		By Theorem \ref{cyclic} the $2$-class group of $L_{n,d}$ is cyclic. By Theorem \ref{lambda} $\lambda^-(L_{0,d})=1$. Thus, $\lambda(L_{0,d})=\lambda^-(L_{0,d})=1$ and the first claim follows. Recall that $\lambda(\widehat{A^-_{\infty}})=\lambda(A^-_{\infty})$. Note that the groups $\widehat{A^-_n}\cap A^+_n$ are of exponent $2$. So if we know that the $2$-class group of $L_{n,d}$ is cyclic and $\lambda(\widehat{A^-_{\infty}})=1$, then $A_n^+$ contains at most $2$ elements. As the capitulation kernel $A_n(L_{n,d}^+)\to A_n(L_{n,d})$ contains at most $2$ elements due to \cite[Theorem 10.3]{washington1997introduction}, we see that the $2$-class group of $L_{n,d}^+$ is uniformly bounded.
 	\end{proof} 
 	We will also need \cite[Theorem 5]{chemsZkhnin2} and \cite[Theorem 1]{chems} which are summarized in the following Theorem.
 	\begin{Theorem}
 		\label{thmrank2}Let $n\geq 1$ and assume that $d$ takes one of the following forms:
 		\begin{enumerate}[\rm 1.]
 			\item $d=pq$, for two primes $p$ and $q$   congruent to $3\pmod 8$.
 			\item $d=p$, is a prime congruent to $9\pmod{16}$.
 		\end{enumerate}
 		Then the rank of the $2$-class group of $L_{n,d}$ is $2$.
 	\end{Theorem}
 	

 	\section{{ The Main results}}
 	
 	\begin{Lemma}
 		\label{lemma}
 		Let $d$ be a square-free integer. We have:
 		\begin{enumerate}[\rm $1.$]
 			\item  $h_2(L_{1,d})=2\cdot h_2(-d)$,
 			if $d=pq$, for two primes $p\equiv 5\pmod 8$ and $q\equiv 3\pmod 8$.
 			\item $h_2(L_{1,d})=h_2(-2d)$, if $d=pq$, for two primes $p\equiv q\equiv 3\pmod 8$ or $d=p$ for a prime  $p$ such that $p\equiv 9\pmod{16}$ and $(\frac{2}{p})_4=1$.
 			
 		\end{enumerate}
 	\end{Lemma}
 	\begin{proof}
 		\begin{enumerate}[\rm $1.$]
 			\item 
 		Suppose that $d$ takes the first form of the lemma.  By \cite[Corollary 19.7]{connor88}   $h_2(pq)=h_2(2pq)=2$ and by \cite[p. 353]{kaplan76} $h_2(-2qp)=4$.
 		Denote by $\varepsilon_{2pq}$ the fundamental unit of the quadratic field $\mathbb Q(\sqrt{2pq})$. We have  $\varepsilon_{2pq}=x+y\sqrt{2pq}$, for some integers $x$ and $y$. Since $\varepsilon_{2pq}$ has a positive norm we obtain $x^2-2pqy^2=1$. Thus $x^2-1= 2pqy^2$. Put $y=y_1y_2$ for $y_i\in \mathbb{Z}$. For a moment, let us suppose that we can write
 		$$\left\{ 
 		\begin{array}{ccc}
 		x\pm1&=&y_1^2\\
 		x\mp1&=&2pqy_2^2.
 		\end{array}\right. $$
 		Hence 
 		$1=\left(\frac{y_1^2}{p}\right)=\left(\frac{x\pm1}{p}\right)=\left(\frac{x\mp1\pm 2}{p}\right)=\left(\frac{\pm 2}{p}\right)=\left(\frac{ 2}{p}\right)=-1,$
 		which is impossible. So $x\pm 1$ is not square in $\mathbb{N}$. So from  the third and the fourth item of  \cite[Proposition 3.3]{AZT2016}, we deduce that $q(L_{1,d})=4$. By the class number formula $($cf. \cite[p. 201]{wada}$)$, we have 
 		\begin{eqnarray*}
 			h_2(L_{1,d})	&=&\frac{1}{2^5}q(L_{1,d})h_2(pq) h_2(-pq)h_2(2pq)h_2(-2qp)h_2(2)h_2(-2)h_2(-1)\\
 			&=&\frac{1}{2^5}q(L_{1,d})h_2(pq) h_2(-pq)h_2(2pq)h_2(-2qp)\\
 			&=&\frac{1}{2^5}\cdot4\cdot2\cdot h_2(-pq)\cdot2\cdot4  \\
 			&=&2 \cdot h_2(-pq).
 		\end{eqnarray*}
 	
 	\item Suppose now that $d$ takes one of the  forms in the second item.	  Then we have the result by    \cite[Corollary 2]{ACZ2} and   \cite[The proof of Theorem 1, p. 7]{ACZ2}.	\end{enumerate}
 	\end{proof}

 	\begin{Theorem}\label{1st main thm}
 		Let $d$ be in one of the following cases:
 		\begin{itemize}
 			\item $d=p$ be a prime congruent to $9\pmod{16}$ and assume that $(\frac{2}{p})_4=1$.
 			\item $d=pq$ for two primes congruent to $3 \pmod{8}$.
 		\end{itemize}
 		Let $2^r=h_2(-2d)$.
 		Then for $n\ge 1$ the $2$-class group of $L_{n,d}$ is isomorphic to the group $\mathbb{Z}/2\mathbb{Z}\times\mathbb{Z}/2^{n+r-2}\mathbb{Z}$. In the projective limit we obtain $\ZZ_2\times \ZZ/2\ZZ$.
 	\end{Theorem}
 	\begin{proof}
 		By Theorem \ref{thmrank2} we know that the $2$-rank of the $2$-class group of $L_{n,d}$ equals $2$ for $n\ge 1$. Further $\lambda^-(L_{0,d})=1$ due to Theorem \ref{lambda} and $h_2(L_{1,d})=2^r$ by Lemma \ref{lemma}. By Theorem \ref{thm 4} the class number of $L^+_{n,d}$ is odd for all $n$. As there is no capitulation in $A_n^-$ (see \cite[Lemma 2.2]{Katharina}) and $\lambda^-(L_{0,d})=1$ we see that $A_n^-$ has rank one for $n$ large enough (see also Lemma \ref{lem:rankbounded}). That implies that the second generator of the $2$-class group of $L_{n,d}$ is a class of a ramified prime in $L_{n,d}/L^+_{n,d}$. As the class number of $L_{n,d}^+$ is odd these ramified classes have order $2$ and we obtain that the $2$-class group of $L_{n,d}$ is isomorphic to $\mathbb{Z}/2\mathbb{Z}\times\mathbb{Z}/2^{l_n}\mathbb{Z}$.
 		
 		Let $E$ be the elementary $\Lambda$-module associated to $A_{\infty}$. Then according to \cite[page 282-283]{washington1997introduction}  $\nu_{n,0}E=2\nu_{n-1,0}E$ for all $n\ge2$: Indeed, 
$\nu_{n,0}=\nu_{n,n-1}\nu_{n-1,0}$. As $E$ has $\ZZ_2$-rank $1$ we know that $T$ acts as $2v$ on $E$ for some $v\in \ZZ_p$. For $n\ge 2$ we have $\nu_{n,n-1}=(T+1)^{2^{n-1}}-1+2$. The term $(T+1)^{2^{n-1}}-1=T^{2^{n-1}}+O(2T)$  acts as $4v'$ on $E$ for some $v'\in \ZZ_2$. Hence, $\nu_{n,n-1}$ acts as $2u$ on $E$ for some unit $u\in \ZZ_2$ and all $n\ge 2$.
 		 Hence, \[\vert E/\nu_{n,0}E\vert=\vert E/2^{n-1}E\vert \vert E/\nu_{1,0}E\vert=2^{n-1+c'}\] for $n\ge 1$ and some constant $c'\ge 1$ independent of $n$. Note that we can rewrite this as $ \vert E/\nu_{n,0}E\vert=2^{n+c}$. As $E$ has only one $\ZZ_2$-generator we can assume that the pseudoisomorphism $\phi: A_{\infty}\to E$ is surjective. The maximal finite submodule of $A_{\infty}$ is generated by the classes $(c_n)_{n\in \mathbb{N}}$ of the ramified primes above $2$. Let $\tau$ be a generator of $Gal(L_{d,\infty}/L_{0,d})$. Then $\tau(c_n)=c_n$ as the primes above $2$ are totally ramified in $L_{\infty,d}/\mathbb{Q}(\sqrt{d})$. It follows that $Tc_n=0$. Hence, for every $n\ge 1$ the kernel of $\overline{\phi}: A_{\infty}/\nu_{n,0}A_{\infty}\to E/\nu_{n,0}E$ is isomorphic to the maximal finite submodule in $A_{\infty}$ and contains $2$ elements. 
 		Let $Y$ be such that $A_{\infty}/Y\cong A_0$. Then $A_n\cong A_{\infty}/\nu_{n,0}Y$ \cite[page 281]{washington1997introduction}. Then we obtain \[\vert A_n\vert =\vert A_{\infty}/\nu_{n,0}Y\vert=\vert A_{\infty}/\nu_{n,0}A_{\infty}\vert \vert \nu_{n,0}A_{\infty}/\nu_{n,0}Y\vert=2^{n+c+1}\vert \nu_{n,0}A_{\infty}/\nu_{n,0}Y\vert \textup{ for } n\ge 1. \]
 		As the maximal finite submodule in $A_{\infty}$ is annihilated by $\nu_{n,0}$ we see that the size of $\nu_{n,0}A_{\infty}/\nu_{n,0}Y$ is constant independent of $n$.
 		Hence, we obtain that the $2$-class group of $L_{n,d}$ is of size $2^{n+\nu}$ for all $n\ge 1$. Using that $h_2(L_{1,d})=2^r$ we obtain $\nu=r-1$. This yields $2\cdot 2^{l_n}=2^{n+r-1}$ and we obtain $l_n=n+r-2$. Noting that $L_{n,d}$ is the $n$-th step of the field $L_{0,d}$ finishes the proof of the first claim. As the direct term $\ZZ/2\ZZ$ is norm coherent the second claim is immediate.
 	\end{proof}

 	\begin{corollary} Let $d$ be in one of the following cases:
 		\begin{enumerate}[\rm $\bullet$]
 			\item $d=p$ a prime congruent to $9\pmod{16}$ and assume that $(\frac{2}{p})_4=1$,
 			\item $d=pq$ for two primes congruent to $3 \pmod{8}$.
 		\end{enumerate}
 		If $d$ takes the first form set $p=u^2-2v^2$ where $u$ and $v$ are two positive integers  such that $u\equiv 1\mod 8$.\\
 		If $d$ takes the second form set   $\left(\frac{p}{q}\right)=1$ and let the integers    $X,Y,k,l$ and $m$ be such that $2q=k^2X^2+2lXY+2mY^2$ and $p=l^2-2k^2m$ $($see  \cite[p. 356]{kaplan76} for their existence$)$. Let $2^r=h_2(-d)$. For all $n\geq 1$, we have:
 		\begin{enumerate}[\rm 1.]
 			\item If $d$ takes the first from, then  the $2$-class group of $L_{n,d}$ is isomorphic to $\ZZ/2\ZZ\times \ZZ/2^{n+1}\ZZ$  if and only if  $ \left(\frac{u}{p}\right)_4=-1$.\\
 			Elsewhere, it is isomorphic to $\mathbb{Z}/2\mathbb{Z}\times\mathbb{Z}/2^{n+r-2}\mathbb{Z}$,  where $r\geq 4$ was defined in Theorem \ref{1st main thm}.
 			\item If $d$ takes the second from, 
 			then  the $2$-class group of $L_{n,d}$ is isomorphic to $\ZZ/2\ZZ\times \ZZ/2^{n+1}\ZZ$  if and only if  $ \left(\frac{-2}{|X+lY|}\right)=-1$.

 			Elsewhere, it is isomorphic to $\mathbb{Z}/2\mathbb{Z}\times\mathbb{Z}/2^{n+r-2}\mathbb{Z}$, where $r\geq 4$ was defined in Theorem \ref{1st main thm}.
 		\end{enumerate}
 	\end{corollary}
 	\begin{proof}By Lemma \ref{lemma} we know	$h_2(L_{1,d})=h_2(-2d)$.
 		Since the $2$-rank of $\mathrm{Cl}(L_{1,d})$  equals $2$ and $\vert \mathrm{Cl}(L_{n,d})\vert \neq 4$ (see \cite[Theorem 5.7]{ACZ}) it follows that  $h_2(-2d)$ is divisible by $8$.  Thus  \cite[Theorem 2]{leonard1982divisibilityby16} (resp. \cite[pp. 356-357]{kaplan76}) gives the first (resp. second) item.
 	\end{proof}
 	
 	We give the following numerical examples that illustrating the above corollary:
 	\begin{enumerate}
 		\item Set $p=89$, $u=17$ and $v=10$. We have $p=u^2-2v^2$ and  $\left(\frac{2}{p}\right)_4=-\left(\frac{u}{p}\right)_4=1$.  So the $2$-class group of $L_{n,d}$ is isomorphic to $\ZZ/2\ZZ\times \ZZ/2^{n+1}\ZZ$, for all $n\geq 1$.
 		\item Let $p=11$, $q=19$, $k=1$, $l=3$, $m=-1$, $X=4$ and  $Y=1$. We have :  $p=l^2-2k^2m$ and $2q=k^2X^2+2lXY+2mY^2$. Since  $\left(\frac{-2}{\left|X+lY\right|}\right)=\left(\frac{-2}{7}\right)=-1$, So the $2$-class group of $L_{n,p}$ is isomorphic to $\ZZ/2\ZZ\times \ZZ/2^{n+1}\ZZ$, for all $n\geq 1$.
 	\end{enumerate}

 	\begin{Theorem}\label{thm 2}
 		Assume that $d=pq$ is the product of two primes  $p\equiv-q\equiv 5  \pmod 8$  and $2^r=2\cdot h_2(-pq)$. Then for $n\ge 1$ the $2$-class group of $L_{n,d}$ is isomorphic to $\mathbb{Z}/2^{n+r-1}\mathbb{Z}$.
 	\end{Theorem}
 	\begin{proof}
 		We know already from Theorem \ref{cyclic} that the $2$-class group of $L_{n,d}$ is cyclic and by Theorem \ref{thm 3} that $\lambda(L_{0,d})=1$. In particular, the module $A_{\infty}$ does not contain a finite submodule and is hence isomorphic to its elementary module $E$. Let $Y$ as in the proof of Theorem \ref{1st main thm}, then there is no $\nu_{n,0}$-torsion and we obtain that the size of $\nu_{n,0}A_{\infty}/\nu_{n,0}Y$ equals a constant independent of $n$. As before we obtain $\vert A_n\vert=\vert A_{\infty}/\nu_{n,0}A_{\infty}\vert \vert \nu_{n,0}A_{\infty}/\nu_{n,0}Y\vert=2^{n+d}.$ In particular, Iwasawa's formula holds for all $n\ge 1$. Hence, $h_2(L_{1,d})=2^r=2^{1+\nu}$ and $\nu=r-1$ and the claim follows.
 	\end{proof}
 	
 	\begin{corollary} 
 		\label{cor}Let  $d=pq$   be the product of two primes $p$ and $q$ such that  $p\equiv-q\equiv 5  \pmod 8$. Then for all $n\geq 1$, we have
 		\begin{enumerate}[\rm 1.]
 			\item If $\left(\frac{p}{q}\right)=-1$, then,    the $2$-class group of $L_{n,d}$ is isomorphic to $ \ZZ/2^{n+1}\ZZ$.
 			\item If $\left(\frac{p}{q}\right)=1$ and $\left(\frac{q}{p}\right)_4=1$, then, the $2$-class group of $L_{n,d}$ is isomorphic to $ \ZZ/2^{n+2}\ZZ$.
 		\end{enumerate}	
 		Elsewhere, the $2$-class group of $L_{n,d}$ is isomorphic to $ \ZZ/2^{n+r-1}\ZZ$, where $r\geq 4$ was defined in Theorem \ref{thm 2}.
 	\end{corollary}
 	\begin{proof}
 		By the previous theorem, the first item is direct from  \cite[19.6 Corollary]{connor88} and also the rest is direct from \cite[Theorem 3.9]{Qin} and its proof.
 	\end{proof}
 	
 	We give the following  numerical examples illustrating the above corollary:
 	\begin{enumerate}
 		\item Let $d=13\cdot 19$. We have  $\left(\frac{13}{19}\right)=-1$. So the $2$-class group of $L_{n,p}$ is isomorphic to $ \ZZ/2^{n+1}\ZZ$, for all $n\geq 1$.
 		\item Let $d=5\cdot 11$. We have $\left(\frac{5}{11}\right)=1$ and $\left(\frac{11}{5}\right)_4=1$. The $2$-class group of $L_{n,d}$ is isomorphic to $ \ZZ/2^{n+2}\ZZ$, for all $n\geq 1$.
 	\end{enumerate}

 	\hspace{0.5cm}

 	Let now $X'$, $Y'$ and $Z$ three positive integers verifying the Legendre equation  
 	\begin{eqnarray}\label{eq legendre}
 	pX'^2+qY'^2-Z^2=0
 	\end{eqnarray}
 	And  satisfying 
 	\begin{eqnarray}\label{cond 1}
 	(X',Y')=(Y',Z)=(Z,X')=(p,Y'Z)=(q,X'Z)=1,
 	\end{eqnarray}
 	and 	
 	\begin{eqnarray}\label{cond 2}
 	X' \text{ odd},  Y' \text{ even and } Z\equiv 1\pmod 4.
 	\end{eqnarray}	
 	(see \cite{Williams-pq} for more details)
 	\begin{corollary} Let  $d=pq$   be the product of two primes $p$ and $q$ satisfying $p\equiv-q\equiv 5  \pmod 8$, $\left(\frac{p}{q}\right)=1$ and   $\left(\frac{-q}{p}\right)_4=1$.  Let $X'$, $Y'$ and $Z$ be three positive integers satisfying the equation  $(\ref{eq legendre})$ and the conditions $(\ref{cond 1})$ and 
 		$(\ref{cond 2})$. If $\left(\frac{Z}{p}\right)_4\not=\left(\frac{2X'}{Z}\right)$, then the $2$-class group of $L_{n,d}$ is isomorphic to $\mathbb{Z}/2^{n+3}\mathbb{Z}$. Elsewhere, it is isomorphic to $\mathbb{Z}/2^{n+r-1}\mathbb{Z}$, for $r\geq 5$ defined as in Theorem \ref{thm 2}.
 		
 	\end{corollary}
 	\begin{proof} It is immediate that the assumptions of Corollary \ref{cor} are not satisfied. Hence, $r\ge 4$.
 		The rest follows directly from Theorem \ref{thm 2} and \cite[Theorem 2]{Williams-pq}.
 	\end{proof}

 	Now we close this section with some numerical examples illustrating the above corollary:
 	\begin{enumerate}
 		
 		\item  Let $p=5$, $q=19$ and $d=-pq$. Then $X'=1$, $Y'=2$ and $Z=9$ are solutions of the equation $(\ref{eq legendre})$  verifying the condition $(\ref{cond 1})$ and 
 		$(\ref{cond 2})$. Furthermore, $\left(\frac{9}{5}\right)_4=-\left(\frac{2}{9}\right)=-1$. Thus, the $2$-class group of $L_{n,d}$ is isomorphic to
 		$\mathbb{Z}/2^{n+3}\mathbb{Z}$.
 		
 		\item  Let $p=37$, $q=11$ and $d=-pq=-407$. Then $X'=1$, $Y'=56518$ and $Z=187449$ are solutions of the equation $(\ref{eq legendre})$  verifying the condition $(\ref{cond 1})$ and 
 		$(\ref{cond 2})$. Furthermore, $\left(\frac{187449}{37}\right)_4=\left(\frac{2}{187449}\right)=1$. Thus, the $2$-class group of $L_{n,d}$ is isomorphic to
 		$\mathbb{Z}/2^{n+r-1}\mathbb{Z}$, for some $r\geq5$. Indeed with these settings $r=5$ (see \cite[p. 230]{Williams-pq}).  		
 	\end{enumerate}

 	\section{{Applications}}
 	Note that   \cite[Theorem 2.5]{Katharina} holds for CM-fields containing the fourth root of unity $i$. Therefore we cannot compute  the $2$-class groups of layers of the cyclotomic $\ZZ_2$-extension of imaginary quadratic fields with the same   techniques used in the previous sections. Therefore as   applications of our above results and using class number formulas and some computations  on Hasse's unit index,   	we deduce the structure of the $2$-class groups of the cyclotomic  $\ZZ_2$-extension of   some imaginary quadratic fields. 
 	\begin{Theorem}\label{thm 10}
 		Let $d$ be an odd  positive square free integer and $r$ such that $2^r=h_2(-2d)$. Let $K_{0,d}=\mathbb{Q}(\sqrt{-d})$ and denote by $K_{n,d}$ the $n$-th step of the cyclotomic $\mathbb{Z}_2$-extension of $K_{0,d}$.
 		Suppose that $d$ takes one of the following forms:
 		\begin{itemize}
 			\item $d=pq$, for two primes $p$ and $q$ congruent to $3 \pmod{8}$.
 			\item $d=p$, for a prime   $p\equiv 9\pmod{16}$ such that that $(\frac{2}{p})_4=1$.
 		\end{itemize}
 		Then for all $n\ge 1$ the $2$-class group of $K_{n,d}$ is isomorphic to $\mathbb{Z}/2\mathbb{Z}\times\mathbb{Z}/2^{n+r-1}\mathbb{Z}$. In the projective limit we obtain $\ZZ_2\times \ZZ/2\ZZ$.
 	\end{Theorem}	
 	
 	\begin{proof}Let $K_n=\mathbb{Q}(  \zeta_{2^{n+2}})$ and $K_{n,d}=\mathbb{Q}(  \zeta_{2^{n+2}}+\zeta_{2^{n+2}}^{-1},\sqrt{-d})=K_n^+(\sqrt{-d})$. 	
 		We have   the following field diagram (see Figure \ref{fig:1}):
 		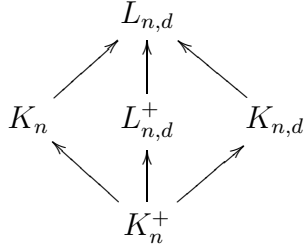
\begin{figure}[H]
 			$$\xymatrix@R=0.7cm@C=0.7cm{
 				&L_{n,d}  \ar@{<-}[d] \ar@{<-}[dr] \ar@{<-}[ld] \\
 				K_n\ar@{<-}[dr]& L_{n,d}^+ \ar@{<-}[d]& K_{n,d}\ar@{<-}[ld]\\
 				& K_n^+}$$
 			\caption{Subfields of $L_{n,d}/K_n^+$.}\label{fig:1}
 		\end{figure}
 		So by class number formula (cf. \cite[Proposition 3 and equation (1)]{lemmermeyer1995ideal}) we have:
 		\begin{eqnarray}
 		h_2(L_{n,d})&=&\frac{Q_{L_{n,d}}}{Q_{K_n}Q_{K_{n,d}}}\cdot\frac{\mu_{L_{n,d}}}{\mu_{K_n}\mu_{K_{n,d}} }\cdot\frac{h_2(K_n)h_2(K_{n,d})h_2(L_{n,d}^+)}{h_2(K_n^+)^2}\nonumber
 		\end{eqnarray}
 		It is known that $h_2(K_n)=1$ and by Theorems \ref{thm 4} and \ref{1st main thm} we respectively have    $h_2(L_{n,d}^+)=1$	and $h_2(L_{n,d})=2^{n+r-1}$. Therefore 
 		\begin{eqnarray}
 		2\cdot 2^{n+r-1}=\frac{Q_{L_{n,d}}}{Q_{K_n}Q_{K_{n,d}}}\cdot h_2(K_{n,d})
 		\end{eqnarray}
 		It is known that   $Q_{K_n}=1$. Let $k=\QQ{(i, \sqrt{d})}$. As the natural norm 		$N_{L_{1,d}/k}: W_{L_{1,d}}/W_{L_{1,d}}^2\rightarrow W_{k}/W_{k}^2 $ is onto, we obtain $Q_{L_{1,d}}$ divides $Q_{k}$ (cf.  \cite[Proposition 1]{lemmermeyer1995ideal}). Since $Q_{k}=1$ (cf. \cite[p. 19]{azizi99unite} and  the proof of \cite[Lemma 4]{ACZ2}), then $Q_{L_{1,d}}=1$. 
 		Since  $N_{L_{n,d}/L_{n-1,d}}: W_{L_{n,d}}/W_{L_n}^2\rightarrow W_{L_{n-1}}/W_{L_{n-1}}^2$ is onto, it follows that $Q_{L_{n,d}}$ divides $Q_{L_{n-1,d}}$. Thus, by induction $Q_{L_{n,d}}=1$. 
 		
 		Note that the extensions $ K_{n,d} $ are essentially ramified (cf. \cite[p. 349]{lemmermeyer1995ideal})   for all $n\geq 1$. Since $\mu_{ K_{n,d} }=2$ we obtain  by \cite[Theorem 1]{lemmermeyer1995ideal} 
 		$Q_{K_{n,d}}=1$. Hence, $h_2(K_{n,d})=2^{n+r}$ and this is for all $n\geq 1$.
 		Since   the rank of the $2$-class group of $K_{1,d}$ equals $2$
 		(cf. \cite[Proposition 4]{mccall1995imaginary})   and the $2$-class group of $K_{n,d}$ is of type $(2,2^{\bullet})$ for $n$ large enough (cf \cite[p. 119]{mizuzawa}), we get the result.

 	\end{proof}

 	Now using   second main theorem of the previous section we will show the next result.
 	\begin{Theorem} 
 		\label{thm11}
 		Assume that $d=pq$ is the product of two primes  $p\equiv-q\equiv 5  \pmod 8$  and $2^r=2\cdot h_2(-pq)$. 
 		Let $K_{0,d}=\mathbb{Q}(\sqrt{-d})$ and denote for $n\ge 1$ by $K_{n,d}$ the $n$-th step of the cyclotomic $\mathbb{Z}_2$-extension of $K_{0,d}$.
 		Then for $n\ge 1$ the $2$-class group of $K_{n,d}$ is isomorphic to $\mathbb{Z}/2^{n+r-1}\mathbb{Z}$. 
 	\end{Theorem}
 	
 	\begin{proof}We keep similar notations and proceed as    in the proof of   Theorem \ref{thm 10}. Note that by \cite[Proposition 3]{ACZ2}, we have $h_2(L_{n,d}^+)=2$. So as above 
 	we show that:
 		\begin{eqnarray}
 		h_2(L_{n,d})&=&\frac{Q_{L_{n,d}}}{Q_{K_n}Q_{K_{n,d}}}\cdot\frac{\mu_{L_{n,d}}}{\mu_{K_n}\mu_{K_{n,d}} }\cdot\frac{h_2(K_n)h_2(K_{n,d})h_2(L_{n,d}^+)}{h_2(K_n^+)^2}\cdot\nonumber 
 		\end{eqnarray}
 		Thus 	
 		\begin{eqnarray}
 		2^{n+r-1}&=&\frac{1}{1\cdot 1}\cdot\frac{2^n}{2^n\cdot 2 }\cdot\frac{1\cdot h_2(K_{n,d})\cdot 2}{1}\cdot\nonumber
 		\end{eqnarray}
 		Thus, $h_2(K_{n,d})=2^{n+r-1}$, for all $n$. Since $L_{n,d}/K_{n,d}$ is ramified, then $2$-rank$(\mathrm{Cl}(K_{n,d}))\leq$ $2$-rank$(\mathrm{Cl}(L_{n,d}))=1$
 		(Theorem \ref{thm 2}).  Which 
 		completes the proof.	
 	\end{proof}
 	
 	\begin{Remark}
 		One can easily deduce analogous corollaries as in the previous section.
 	\end{Remark}


\begin{thebibliography}{9}
 		
 		
 		\bibitem{azizi99unite}   A. Azizi,   {\it Unités de certains corps de nombres imaginaires et abéliens sur $\mathbb{Q}$,}{ Ann. Sci. Math. Québec, {\bf23} (1999),  15-21.}
 		\bibitem{ACZ} A. Azizi,  M. M. Chems-Eddin and A. Zekhnini,  On the rank of the $2$-class group  of some imaginary triquadratic number  fields.  arXiv:1905.01225v3.
 		\bibitem{ACZ2} A. Azizi,  M. M. Chems-Eddin and A. Zekhnini,  On some imaginary triquadratic number fields $k$ with
 		$\mathrm{Cl}_2(k)\simeq(2, 4)$ or $(2, 2, 2)$.  	arXiv:2002.03094.
 		\bibitem{AZT2016} A. Azizi, A. Zekhnini and M. Taous,  On the strongly ambiguous classes of some biquadratic number  fields. \textit{Math. Bohem.,} \textbf{141} (2016), 363–384.	
 		
 		
 		
 		
 		\bibitem{atsuta} M. Atsuta, Finite $\Lambda$-submodules of Iwasawa modules for a $CM$-Field for $p=2$, \textit{J. de Theor. des Nr. de Bordeaux,} \textbf{30} (2018), $1017$-$1035$.		
 		
 		
 		
 		\bibitem{chems}M. M. Chems-Eddin,  The rank of the $2$-class group of some number fields with large degree, Accepted for publication in Algebra Colloquium.	arXiv:2001.00865v2.
 		
 		\bibitem{chemsZkhnin2}M. M. Chems-Eddin, A. Azizi  and A. Zekhnini,   On the $2$-class group of some number fields with large degree.	arXiv:1911.11198.
 		
 		
 		
 		\bibitem{connor88} P. E. Conner and J. Hurrelbrink, { Class number parity}\textit{ Ser. Pure Math.,} vol. \textbf{8}, World Scientific, Singapore (1988).
 		
 		\bibitem{froh}A. Fr\"ohlich, Central Extensions, Galois Groups, and Ideal Class Groups of Number Fields, \textit{Contemporary Mathematics} vol. \textbf{ 24} American Mathematicla Society, Providence (1983).
 		
 		
 		\bibitem{iwasawa59}K. Iwasawa,  On $\varGamma$-extensions of algebraic number fields, \textit{Bull. Amer. Math. Soc.}, \textbf{65} (1959), 183–226.
 		
 		\bibitem{kaplan76} P.~Kaplan,  Sur le $2$-groupe des classes d'id\'eaux des corps quadratiques, \textit{J. Reine Angew. Math.,} \textbf{283-284} (1976), 313–363.		
 		
 		\bibitem{kida} {{Y. Kida,}} Cyclotomic $\ZZ_2$-extensions of J-fields, \textit{J. Number Theory,} \textbf{14} (1982), 340-352. 
 		
 		
 		\bibitem{Williams-pq} { {P. A. Leonard and  K. S. Williams,}} On the divisibility of the class number of $\mathbb{Q}(\sqrt{-pq})$ by 16, \textit{Proc. Edinburgh Math. Soc.,} \textbf{26} (1983), 221-231.
 		
 		
 		\bibitem{leonard1982divisibilityby16} {P. A. Leonard and  K. S. Williams,}  On the divisibility of the class numbers of $\mathbb{Q}(\sqrt{-d})$ and $\mathbb{Q}(\sqrt{-2d})$ by $16$, \textit{ Can. Math. Bull.,} \textbf{25} (1982), 200-206.	
 		
 		\bibitem{mizuzawa} Y. Mizusawa. On the maximal unramified pro-2-extension over the cyclotomic $\ZZ_2$-extension of an imaginary quadratic field, \textit{J. de Theor. des Nr. de Bordeaux}, \textbf{22} (2010), $115$-$138$.
 		
 		
 		\bibitem{LiouYangXuZhang}{J. Li, Y. Ouyang, Y. Xu and S. Zhang,}  $l$-class groups of fields in Kummer towers,  arXiv:1905.04966.	
 		
 		\bibitem{lemmermeyer1995ideal}
 		F.~Lemmermeyer, {\it Ideal class groups of cyclotomic number fields I,}{  Acta Arith., 72 (1995),  347–359.}
 		
 		
 		\bibitem{mccall1995imaginary}	T. M. McCall, C. J. Parry, R. R. Ranalli,  Imaginary Bicyclic Biquadratic Fields with Cyclic $2$-Class Group, J. Number Theory, 53 (1995), 88-99.
 		
 		\bibitem{Katharina} { K. M\"uller}. Capitulation in the $\mathbb{Z}_2$ extension of CM number fields, \textit{Math. Proc. Camb. Phil. Soc.,} \textbf{166} (2019), 371-380. 	
 		
 		
 		
 		\bibitem{wada}{ H. Wada,}  On the class number and the unit group of certain algebraic number fields, \textit{ J. Fac. Sci. Univ. Tokyo,} \textbf{13} (1966),  201--209.
 		
 		
 		
 		
 		\bibitem{washington1997introduction}{{L.~C. Washington,}  Introduction to cyclotomic fields,\textit{  Graduate Texts in Mathematics} \textbf{83}. Springer,  1997.}
 		
 		
 		
 		
 		\bibitem{Qin}	{Q. Yue,} 8-ranks of class groups of quadratic number fields and their densities, \textit{Acta Math. Sin.} \textbf{17} (2011), 1419--1434.
 		
 		
 		\bibitem{Qin2}	{Q. Yue,}  The generalized Rédei-matrix, \textit{Math. Z.} \textbf{261} (2009), 23--37.
 		
 		
 		
 		
 	\end{thebibliography}
 \end{document}